%% file: LarSubPaper.tex
\documentclass[10pt]{amsart}
\usepackage{psfrag,epsfig,amsmath,amsfonts,amssymb,amscd,amsthm,amsbsy,epsf,calc,subfigure}
\newtheorem{thm}{Theorem}[section]
\newtheorem{lem}[thm]{Lemma}
\newtheorem{cor}[thm]{Corollary}

\newtheorem{ex}[thm]{Example}

\newtheorem{DEF}[thm]{Definition}
\theoremstyle{remark}                  
\newtheorem{rem}[thm]{Remark}

\def\ep{\varepsilon}
\def\al{\alpha}

\newcommand{\ds}{\displaystyle}

\begin{document}
\title{ABOUT THE GEOMETRY and REGULARITY of LARGEST SUBSOLUTIONS For a FREE BOUNDARY problem in $\mathbf{R}^2$: Elliptic case}
\author{Betul Orcan}
\begin{abstract}
We study geometric and regularity properties of the largest subsolution of a one-phase free
boundary problem under a very general free boundary condition in $\mathbb{R}^2$. Moreover, we provide density bounds for the positivity set and its complement near the free boundary.
\end{abstract}
\address{Department of Mathematics, University of Texas at Austin, Austin, TX, USA}
\email{borcan@math.utexas.edu}

\maketitle
\baselineskip=16pt
\pagestyle{plain}              


\section{Introduction}\label{sec:intro}

In this paper, we study the geometry and regularity of the largest subsolution
of the following Free Boundary Problem (FBP) for a given bounded open domain $D \subseteq \mathbb{R}^2$ and $u: \mathbb{R}^2\setminus D \mapsto [0,+\infty)$ is a continuous function which satisfies:
\begin{equation}\label{eq:FB}
\begin{cases}

\ds\triangle u =0 &\text{ in } \Omega(u) \backslash D,\\
u=g(x),\;  &\text{ on } \partial D, \\
u=0, |\nabla u|^2= f(x), &\text{ on } \partial \Omega(u),
\end{cases}
\end{equation}
where $\Omega(u)=\{x\in \mathbb{R}^2| u(x)>0\}$; $g(x)$ and $f(x)$ are positive continuous
functions. For $f(x)$, there exist $\Lambda, \lambda >0$ such that $0<\lambda <
f(x)<\Lambda,\; \text{for\; all}\; x \in \mathbb{R}^2.$

There is a wide range of physical models related to the above FBP, encompassing problems such as flame propagation and G-equations, capillary drops on a flat or inclined surface, phase transitions, and obstacle problems. There are previous results about the regularity of variational and weak solutions to these example FBPs; for the variational solutions of problem \eqref{eq:FB}, Alt and Caffarelli have results in \cite{AltCaff81}; for the two-phase problem, Alt et al. in \cite{ACF84}; for the three-dimensional case, Caffarelli et al. in \cite{CaJerKenGlbl}; Caffarelli and Shahgholian in \cite{CaffShah} when $f(x)$ is Lipschitz; viscosity solutions in the two-phase problem were studied by
Lederman and Wolanski in \cite{LederWolanski}; the geometry of the free boundary in terms of the weak solution  Kenig and Toro in \cite{KenigToro}. One can consider the problem \eqref{eq:FB} as a linearized version of the capillary drop problem; for the capillary drop problem in variational case, Caffarelli and Friedman have geometric and regularity results in \cite{CaffFriedman}; the inhomogeneous surface and inclined surface cases were considered by Caffarelli and Mellet, \cite{CaffMellIncldSurf,CaffMellCapllaDrops} . If we consider the evolution problem corresponding to \eqref{eq:FB}, then some examples from the literature are: for the  the heat equation, Caffarelli and V\'{a}zquez in \cite{CaffVaz95}; for the front propagation problem in terms of pulsating wave solutions Berestycki and Hamel, \cite{BerstkiHamel02}. For more references, see the book of Caffarelli and Salsa, \cite{CaffSalsa}. Most of these results require that the Free Boundary Condition (FBC) is at least Lipschitz and the media is periodic, We would like to extend these results both to viscosity solutions and to the random case since real life systems also require to work with these cases.  In this context, an example would be a
linearized version of a drop sliding through an inclined plane with random parallel grooves.
In that case, we expect the leading edge of the drop to be steeper, the
drop getting stuck on the grooves or ``the least supersolution of the free
boundary problem", while the back edge getting hang to the grooves and is
flatter, ``the largest subsolution". The problem for the least supersolution of \eqref{eq:FB} has been studied
extensively under smooth and periodic data, \cite{CaLee07, CaLee05}. Among the reasons for its popularity, one is that because it has much better non
degeneracy properties and is much simpler.
\par For some models, the largest subsolution is the proper object of study. One of our ongoing research is about the homogenization problems of FBPs in stationary ergodic case, \cite{BetOr}. In this problem we use a method for viscosity solutions, at first we tried to work on the least supersolution but it did not work, that is why, we started to consider the largest subsolution. On the other hand, there were no prior regularity results for the largest subsolution and this was the our starting point.
\par In addition to this,
for problems in random media, expecting to have Lipschitz regularity for the given data is not reasonable; even continuity of it may not hold in this case. In the random case, media can be heterogenous without any periodic setting, i.e. the FBC can be at most positive, bounded, and continuous in the space variable. In this paper, we focus on
regularity issues for a FBP related to these phenomena  and we concentrate on the geometric description of the
largest viscosity subsolution in two dimensions (the case of a back edge of a
capillary drop or a flame propagating on a planar region) with weaker requirements on the data. We develop a regularity and non
degeneracy theory for it's largest subsolution and give a geometric characterization of the free boundary. Motivated by the study of random media, we allow for the data to be highly oscillatory. Thus, we only require $f(x)$ to be positive, bounded, and measurable function. We used the continuity of $f(x)$ only to be able to obtain the continuous viscosity solutions. One can weaken the continuity assumption on $f(x)$ by taking into account suitable viscosity solution definitions such as in terms of upper-lower semi-continuous functions, as done in \cite{CrIshiiLions}. Moreover, we cannot generalize these results to $\mathbb{R}^n$ because of the lost of Non-Degeneracy, instead, the order of the lower bound of the growth rate near free boundary point $x_0$ becomes $r^{n-1}$ in $B_r(x_0)$.
 \par The paper is organized as follows: Chapter \ref{Definitions} presents the construction of the largest subsolution of \eqref{eq:FB}. Chapter \ref{ChapLipNonDEg} presents the results of Lipschitz and Non-Degeneracy properties. In Chapter \ref{ChapLocalResults}, we will show the geometric properties of the free boundary. Because of weak assumptions on the FBC, one can expect to have a very unstable free boundary (highly oscillatory) on the contrary Chapter \ref{ChapLocalResults} guarantees us that, locally, the normalized neighborhood of the free boundary has two components with positive densities with one of them is the positivity set and the other one is the zero level set, i.e. locally, free boundary does not have high and irregular oscillation.

\par Our results in this paper are the following:
\begin{thm}
  Let $u$ be the largest subsolution of \eqref{eq:FB}, then $u$ has the following properties:
  \begin{itemize}
    \item[(i)] $u$ is a viscosity solution of \eqref{eq:FB},
    \item[(ii)] $u$ is Lipschitz,
    \item[(iii)] $u$ is Non-Degenerate,
    \item[(iv)] Locally, $\Omega(u)$
has a single component $\Gamma$ with
a positively dense complement.
  \end{itemize}
\end{thm}
\begin{proof}
  Proof of the theorem will be given separately in next sections starting with Section \ref{ChapLipNonDEg}.

\end{proof}

\section{On Definitions of Viscosity Solutions }\label{Definitions}

In this section, we review the definitions of viscosity subsolution, supersolution, and
solution of \eqref{eq:FB} as in \cite{CaLee07}.
\begin{DEF}\label{Def1} $u$ is a viscosity subsolution \index{viscosity solution @\emph{viscosity solution}! subsolution @\emph{ subsolution}} of \eqref{eq:FB}
 if $u$ is a continuous function in $\mathbb{R}^2\setminus D$ and satisfies the
following conditions:
\begin{enumerate}
  \item $ \triangle u \geq 0, \; \text{in} \; \Omega(u) \backslash D,$
  \item $u \leq g(x)$, on $\partial D,$
  \item \emph{Free Boundary Condition(FBC):} If $x_0 \in  \partial \Omega(u)$ has a tangent ball from
outside of $\Omega(u)$, then  $|\nabla u(x_0)|^2 \geq f(x_0);$ that is, for $\nu$ is the
normal unit vector inward to $\Omega(u)$ at $x_0$,
if $u(x)\leq
\al \langle x-x_0,\nu \rangle^+ + o(|x-x_0|)$ in a neighborhood of $x_0$, then $\al \geq \sqrt{f(x_0)}$.
 \end{enumerate}

\end{DEF}

\begin{DEF}\label{Def2} $u$ is a viscosity supersolution \index{viscosity solution @\emph{viscosity solution}! supersolution @\emph{supersolution}} of \eqref{eq:FB}
 if $u$ is a continuous function in $\mathbb{R}^2\setminus D$ and satisfies the
following conditions:
\begin{enumerate}
  \item $ \triangle u \leq 0, \; \text{in} \; \Omega(u),$
  \item $u \geq g(x)$, on $\partial D,$
  \item  \emph{Free Boundary Condition(FBC):} If $x_0 \in  \partial \Omega(u)$ has a tangent ball from
inside of $\Omega(u)$, then  $|\nabla u(x_0)|^2 \leq f(x_0);$ that is,
for $\nu$ is the
normal unit vector inward to $\Omega(u)$ at $x_0$, if $u(x)\geq
\al \langle x-x_0,\nu \rangle^+ + o(|x-x_0|)$ in a neighborhood of $x_0$, then $\al \leq \sqrt{f(x_0)}$.
 \end{enumerate}
Moreover, $u$ is a viscosity solution \index{viscosity solution @\emph{viscosity solution}} of \eqref{eq:FB} if it is both a
viscosity sub- and supersolution of \eqref{eq:FB}.
\end{DEF}
Heuristically,  Definitions \ref{Def1} and \ref{Def2} imply the following facts for a viscosity solution, $v$, of \eqref{eq:FB}:
\begin{itemize}
  \item By Definition \ref{Def1}: If $x_0 \in  \partial \Omega(v)$ has a tangent ball from
outside of $\Omega(v)$, then whenever $v(x)$ is touched by a plane from above at $x_0$, then the slope of the plane should be at least $ \sqrt{f(x_0)}$.
  \item By Definition \ref{Def2}: If $x_0 \in  \partial \Omega(v)$ has a tangent ball from
inside of $\Omega(v)$, then whenever $v(x)$ is touched by a plane from below at $x_0$, then the slope of the plane should be at most $ \sqrt{f(x_0)}$.
\end{itemize}

From now on, sub- or super-solution of an equation mean being a sub- or super solution in the viscosity sense.
Here is an example for a simple case in order to see the heuristic picture:

\begin{ex}\label{Exm1}
 If we consider \eqref{eq:FB} with the conditions as $D=B_1(0)$, $g(x)\equiv 1$ and $f(x)\equiv 2$, then a viscosity solution can be obtained by a chopped-up harmonic (subharmonic) function in $\mathbb{R}^2 \setminus B_1(0)$ such as $w(x)=\ds( \frac{-1}{\ln(R)}\ln(\frac{|x|}{R}))_+$, shown in Figure \ref{HarmnResim}, where $R>1$ satisfies
 $|\nabla w(x)|^2= \ds \frac{1}{R^2\ln^2(R)}=2 \; \text{on} \; \partial B_R(0).$
\begin{figure}
  \includegraphics[width=2.5in]{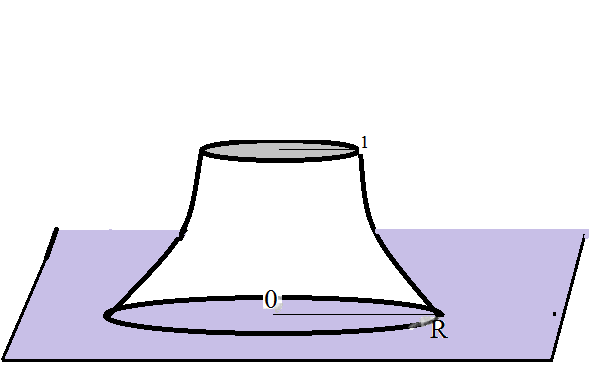}\\
  \caption{w(x) in Example \ref{Exm1}}\label{HarmnResim}
\end{figure}

\end{ex}

 In order to obtain the largest subsolution of \eqref{eq:FB}, we use Perron's method for viscosity solutions (See \cite{CrIshiiLions}). For a fixed supersolution $v$, if we take the largest subsolution $u$ which is smaller than $v$, then we obtain a solution by Perron's method. Thus, let us construct a supersolution of \eqref{eq:FB} by taking a
suitable
harmonic function. Without loss of generality, we assume that $0 \in D$. Then, let us take two balls $B_{R_0}(0)$ and $B_r(0)$ both containing $D$. Let $h(x)=\ds (\frac{1}{\ln(r/R_0)}\ln(\frac{|x|}{R_0}))_+$ be the subharmonic function in $\mathbb{R}^2\setminus D$, then choose $R_0$ as
 $$R_0=\ds \inf\{R| R>r \;\text{and}\;\ds \frac{1}{R^2 \ln^2(r/R)} <\lambda \}.$$ Then, we get $|\nabla h(x)|^2= \ds \frac{1}{R_0^2 \ln^2(r/R_0)} =\lambda$ on $\partial B_{R_0}(0)$. Hence, $h(x)$ is a supersolution of \eqref{eq:FB}.  Therefore, any subsolution should be smaller than $h$ by the comparison principle.
From now on, let us denote $$u= \sup \{v\in
\mathcal{C}(\mathbb{R}^2\setminus D)| v \leq h \;\text{ and v is a
subsolution of \eqref{eq:FB}  } \}.$$
By Perron's Method, $u$ is a viscosity solution of \eqref{eq:FB} and it is the largest subsolution \index{viscosity solution @\emph{viscosity solution}! largest subsolution @\emph{ largest subsolution}} of \eqref{eq:FB}.
\section{Lipschitz and Non-Degeneracy Properties}\label{ChapLipNonDEg}
In this section, we focus on the regularity properties of the largest subsolution, $u$, of \eqref{eq:FB}. Lipschitz regularity implies a uniform bound on the gradient of $u$ and this property is valid in any dimension. On the other hand, the Non-Degeneracy property is restricted to two-dimensional case and this is one of the main difficulties for the subsolution theory. In higher dimensions, $\mathbb{R}^n$, one can obtain the order of $r^{n-1}$ in $B_r(x_0)$ for the lower bound of the growth rate near a free boundary point $x_0$, i.e. $\ds\sup_{B_r(x_0)} u \geq c r^{n-1}$. In $\mathbb{R}^2$, this corresponds to nontrivial linear growth rate.
\subsection{Lipschitz Property}
\begin{thm}\label{LipscThm}
 $u$ is Lipschitz \index{Lipschitz @\emph{Lipschitz}} .
\end{thm}
\begin{proof} Since this is a one-phase problem, it is enough to show that
$u(x)\leq C d(x,\partial \Omega(u)),$ for some universal constant $C>0$
and every
$x\in \Omega(u).$ Because of the local estimates on derivatives for harmonic functions, we have
$$|\nabla u (x_0)| \leq \frac{C_1}{r^3}\|u\|_{L_1(B_r(x_0))}.$$
If we can show that $u(x)\leq C d(x,\partial \Omega(u)),$ then we will obtain $|\nabla u (x_0)| \leq C_1C,$ for some universal constant $C>0$. Actually, this result implies more than Lipschitz property, instead we obtain a uniform bound on the gradient of $u$. Let us prove that $u(x)\leq C d(x,\partial \Omega(u)),$ for some universal constant $C>0$.
\\By way of contradiction: Assume that there exists a point $x_0 \in \Omega(u) $ with
$d(x,\partial \Omega(u))=1 $ such that $u(x_0)>M$, for some large $M.$ Since
$u$ is harmonic in
$\Omega(u)$ and positive, by the Harnack inequality, we have $\ds
\inf_{B_{1/2}(x_0)}u(x) \geq cM,\; \text{for some}\; c>0.$
Let us define $w(x)$, shown in Figure \ref{ResimLipschtz},
$$w(x)= \left\{\begin{array}{rl}
                 cM(\dfrac{\ln (|x-x_0|/(1+\ep))}{\ln(1/2)})_+, & x\in
\mathbb{R}^2\setminus B_{1/2}(x_0) \\\\
                 cM, & x\in B_{1/2}(x_0)
               \end{array}
\right.$$
then $w$ is
harmonic in
$B_{1+\ep}(x_0)\setminus B_{1/2}(x_0)$.
\begin{figure}
  \includegraphics[width=2.5in]{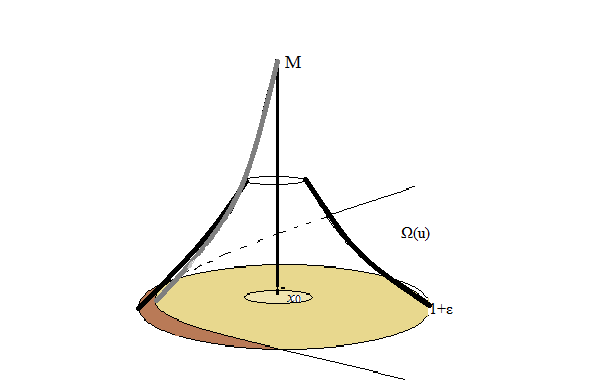}\\
  \caption{w(x) in the proof of Theorem \ref{LipscThm}}\label{ResimLipschtz}
\end{figure}

Consider $v=\max(u,w)$, we claim that it is a subsolution of \eqref{eq:FB} larger than $u$ if we choose $M$ large enough. $v$ is a harmonic function in $\Omega(v)\backslash D$ and $v =u \leq g$ on $\partial D$, so we only need to show that it satisfies the FBC, $(iii)$ in Definition \ref{Def1}. Let $y_0 \in \partial \Omega(v)$ and it has a tangent ball from outside of $\Omega(v)$. If $y_0 \in \partial \Omega(u)$, then $u$ already satisfies the FBC, so does $v$. If $y_0 \in \partial \Omega(w)$, then $|\nabla w|^2(y_0)= \ds [\frac{CM(1+\ep)}{\ln(1/2)}]^2 $.
We can take $M$ as large as we wish in order to make $|\nabla w|^2(y_0)= \ds [\frac{CM(1+\ep)}{\ln(1/2)}]^2 > \Lambda$ which implies, again, $v$ is a
subsolution of
\eqref{eq:FB}. Moreover, $v$ is larger than $u$ which
contradicts to $u$ being
 the largest subsolution. Hence, the result follows.
\end{proof}
\subsection{Non-Degeneracy Property}\label{SEctionNOnDeg}
 Notational Comment: We write $A \sim B$ to mean that, for some universal constants $m, M>0$,
$m\cdot B \leq A \leq M\cdot B.$
\par We will prove that $u$ is Non-Degenerate, i.e. there exists a universal constant
$\kappa>0$  such that $\ds \sup_{B_r(x_0)} u(x)\geq
\kappa r$, for every $x_0 \in \partial\Omega(u),$ by going rigorously through heuristic observations:

  \par $\diamond$ For $x_0 \in \partial\Omega(u)$, we estimate $u(x)$ in $B_r(x_0)$ by Green's representation theorem:
 $$ \begin{array}{ccc}
 0= u(x_0) & = & \ds\int_{\partial B_{r}(x_0)}u G_\nu ds + \ds\int_{
B_{r}(x_0)}(G  \Delta u)dx,
\end{array}
$$
where $G(y,x_0)= \ds\frac{1}{2\pi} \ln \ds(\frac{1}{r}|y-x_0|)$ with $
G \equiv 0$ on $\partial B_{r}(x_0)$. Thus
 $$ \begin{array}{ccc}
-\ds\int_{\partial B_{r}(x_0)}u G_\nu ds &=&  \ds\int_{
B_{r}(x_0)}(G  \Delta u)dx.
\end{array}
$$
  Therefore, we will estimate $u(x)$ in $\partial B_r(x_0)$ by $\ds\int_{
B_{r}(x_0)}  \Delta u dx,$ i.e. by the total mass of $\Delta u$ in $B_r(x_0)$. \par First, we will show that, for the normalized problem:
\par $\diamond$ The total
mass of $\Delta u \sim 1$ in $B_1(0)$ if we assume that $0 \in \partial \Omega (u)$. In order to prove this, we need to show that there exist some constants $c,C>0$ such that
$c \leq \ds\int_{
B_{1}(0)}  \Delta u dx \leq C.$
\begin{itemize}
\item For the upper bound of $\ds\int_{ B_1(0)}(\Delta u)dx$:
\par Since $u$ is Lipschitz in $B_1(0)$, by the Divergence theorem, we have
$$\ds\int_{ B_1(0)}(\Delta u)dx = \ds\int_{\partial B_1(0)}u_\nu dx,$$
we estimate $u_\nu$ by the first order incremental quotient for $x_0 \in \partial B_1(0)$ and obtain the upper bound .
 \item The proof of the lower bound is very technical but the idea is the following:
 \begin{itemize}
   \item $u$ is harmonic in $B_1(0)\cap  \Omega (u)$ so $\ds\int_{ B_1(0)}(\Delta u)dx$ will be nonzero only on $B_1(0)\cap  \partial\Omega (u)$.
   \item If we can show that there exists a partition of the interval $(0,1)$ with some mutually disjoint intervals of the form $(y_0-r_y,y_0+r_y)$ where $\ds\int_{ B_{r_y}(y_0)}(\Delta u)dx \geq c r_y$, then we get $$\ds\int_{ B_1(0)}(\Delta u)dx \geq \ds \sum_{\text{the partition of}\; (0,1)}\ds\int_{ B_{r_y}(y_0)}(\Delta u)dx \geq \sum_{\text{the partition of}\; (0,1)} c r_y= c/2. $$
   \item Thus, in order to obtain the above partition, firstly, we show that
 \begin{itemize}
 \item Lemma: If $x_0\in \overline{\Omega(u)}$ with $d=d(x_0,\partial \Omega(u))$, then
for any $d<r<1$, we have $\partial B_r(x_0)\cap\{u=0\}^o\neq \emptyset$. Since $0 \in \partial \Omega(u)$, this Lemma implies that for any $r \in (0,1)$, we have $\partial B_r(0)\cap\{u=0\}^o\neq \emptyset$. Thus, we can obtain a tangent ball from outside to $\Omega(u)$ for any $r \in (0,1)$. These tangent ball radii will be our candidate partition elements.
 \item Theorem: If $x_0 \in \partial\Omega(u)$ has a
tangent ball from outside of $\Omega(u)$, say $B_r(y)\subseteq \Omega^c(u)$, then $u$ grows linearly in $B_r(x_0)$. This theorem will imply the lower bound of $\triangle u$ in $B_r(x_0)$.

\end{itemize}\end{itemize}

\end{itemize}
We shall show these in detail in a series of lemmas.

\begin{lem}\label{lemnonemptynter}
 Let $x_0\in \overline{\Omega(u)}$ with $d=d(x_0,\partial \Omega(u))$, then we have
$\partial B_r(x_0)\cap\{u=0\}^o\neq \emptyset$, for any $d<r< 1$, as shown in Figure \ref{ResimLemma32}.

\begin{figure}
\includegraphics[width=2in]{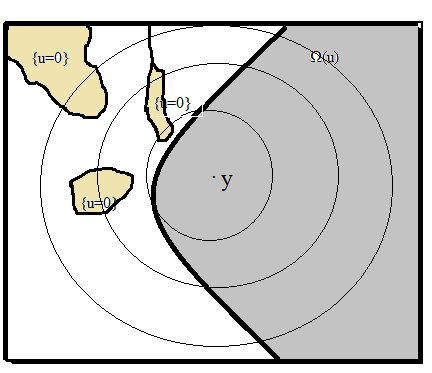}\\
  \caption{Lemma \ref{lemnonemptynter}}\label{ResimLemma32}
\end{figure}
\end{lem}
\begin{proof}
 (By way of contradiction) Assume that there exists $r_0$ such that $\partial B_{r_0}(x_0)\cap\{u=0\}^o=
\emptyset$. Let us define the harmonic function
$h(x)$
as
\begin{equation*}
\begin{cases}
 \triangle h =0 &\text{ in}\; B_{r_0}(x_0), \\
h= u &\text{ in }\;\mathbb{R}^2\setminus B_{r_0}(x_0).\\
\end{cases}
\end{equation*}
By the Maximum principle, $h(x)>0$ in $B_{r_0}(x_0)$ and it is actually a
subsolution: If $h(y)=0$, then $u(y)=0$. If $y\in \partial\Omega(h)$ and has a tangent ball from outside of $\Omega(h)$, then $y\in \partial\Omega(u)$ and since $\partial B_{r_0}(x_0)\cap\{u=0\}^o=
\emptyset$ it has a tangent ball from outside of $\Omega(u)$.
If $h(x)\leq
\al \langle x-y,\nu \rangle^+ + o(|x-y|)$ in a neighborhood of $y$, then we have $u(x) \leq h(x)\leq
\al \langle x-y,\nu \rangle^+ + o(|x-y|)$ so $\al \leq f(y)$ by the FBC, $(iii)$ in Definition \ref{Def1}, satisfied by $u$. This implies $w=\max\{u,h\}$ is a subsolution of \eqref{eq:FB} which is larger than $u$. Contradiction for $u$ being the largest subsolution. Hence, the result follows.
\end{proof}
\begin{thm}\label{linearGrowth} Let $x_0 \in \partial\Omega(u)$ with a
tangent ball from outside of $\Omega(u)$, say $B_r(y)\subseteq \Omega^c(u)$, then $u$ grows linearly \index{
Linear Growth @\emph{Linear Growth}}in $B_r(x_0)$; that is, there exist universal constants $C_1, C_2>0$ such that
$$ C_1 r \leq \sup_{B_r(x_0)} u\leq C_2 r.$$
\end{thm}
\begin{proof}
(By way of contradiction) Assume that $x_0 \in \partial\Omega(u)$ with a tangent ball $B_r(y)\subseteq \Omega^c(u) $ from
outside of $\Omega(u)$ and $u$ does not grow linearly in $B_r(x_0)$. Since $u$ is Lipschitz, we already have $u \leq C_2r $ in $B_r(x_0)$, for some $C_2>0$. Hence, suppose that the first inequality is not true, then there exists $\delta>0$ sufficiently small such that
$\sup_{B_{r}(x_0)} u \leq \delta r.$
Let us define the harmonic function
$$h(x)= \ds\frac{2\delta r}{\ln 2}\ln \dfrac{|x-y|}{r} \; \text{in}\; B_r(y)^c.$$ Shown in Figure \ref{Resimlineargrowth}.
\begin{figure}
 \includegraphics[width=2.5in]{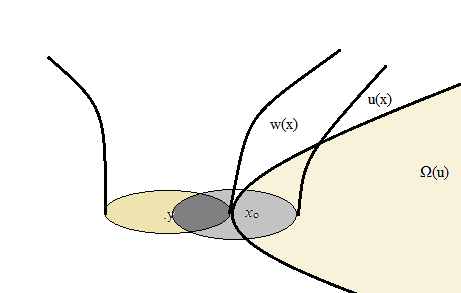}\\
  \caption{$w(x) \geq u(x)$ in $B_r(x_0)\cap \Omega(u)$}\label{Resimlineargrowth}
\end{figure}
Then, $h(x)\geq
u(x)$ in  $B_r(x_0)\cap \Omega(u)$. We can choose $\delta>0$ small enough so that $|\nabla h|^2(x)=[ \ds\frac{2\delta}{\ln 2}]^2< \lambda$ on $\partial B_r(y)$. Then, $u(x) \leq \ds\frac{2\delta}{\ln 2}\langle x-x_0, \nu \rangle + o(|x-x_0|)$ in a neighborhood of $x_0$. Hence, $[\ds\frac{2\delta}{\ln 2}]^2 \geq f(x_0)$ so that we get a contradiction:
 $$\lambda > [\ds\frac{2\delta}{\ln 2}]^2 \geq f(x_0)> \lambda.$$  Hence, there exists $C_1>0$, universal, such that
$ C_1 r \leq \sup_{B_r(x_0)} u.$
\end{proof}

\begin{lem}\label{LemLapRinB} The total
mass of $\Delta u \sim 1$ in $B_1(0)$.
\end{lem}
\begin{proof}
First of all, the total
mass of $\Delta u$ in $B_1(0)$ is bounded by above since $u$ is Lipschitz in $B_1(0)$, by the Divergence Theorem, we have
$$\ds\int_{ B_1(0)}(\Delta u)dx = \ds\int_{\partial B_1(0)}u_\nu dx,$$
we estimate $u_\nu$ by the first order incremental quotient for $x_0 \in \partial B_1(0)$ as
$$\begin{array}{lcl}
   \ds  \frac{u(x_0+ s\nu)-u(x_0)}{s}&\leq& \ds  \frac{C|s \nu|}{s} \leq C,
   \end{array}
$$
where $C>0$ is the universal Lipschitz constant of $u$. Therefore,
$$\ds\int_{ B_1(0)}(\Delta u)dx = \ds\int_{\partial B_1(0)}u_\nu ds \leq 2\pi C.$$

Next, we shall determine the lower bound for the total
mass of $\Delta u $ in $B_1(0)$.
Since $0\in \partial\Omega(u)$, by Lemma \ref{lemnonemptynter}, we have  $\partial B_r(0)\cap \{x| u(x)=0\}^o\neq \emptyset$ for any $r \in (0,1)$. Hence, for each radius $r<1$, there exists a ball $B_{\ep_r}(x_r) \subseteq \{x| u(x)=0\}^o$ and it is tangent to $\partial\Omega(u)$ at some point $y_r \in \partial\Omega(u)$. Pick a ray in $B_1(0)$, say  $R_1=\{r \eta | r\in[0,1], \eta \in S^1\}$. Let us use $R_1$ in order to pick a partition of $(0,1)$ in terms of $\ep_r$. Consider $\ds \bigcup_{r\in[0,1]}\{B_{\ep_r}(y_r)\}$ and rotate all these balls until their centers intersects with $R_1$, i.e. each
$B_{\ep_r}(y_r)$ will be transferred to another ball whose center is on $R_1$ and this center point has the length of $|y_r|$ , as shown in Figure \ref{ResimLaplace}.
\begin{figure}
  \includegraphics[width=2.5in]{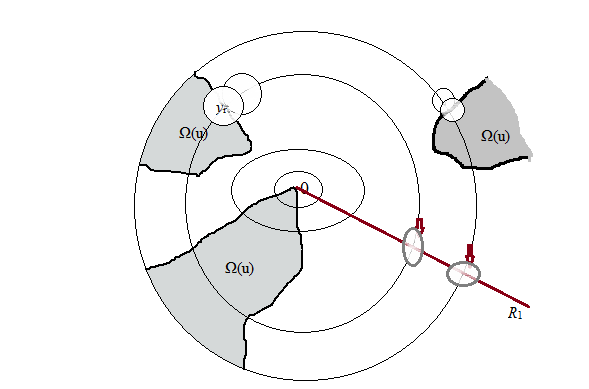}\\
  \caption{Project ball $B_{\ep_r}(y_r)$ onto $R_1$}\label{ResimLaplace}
\end{figure}
Hence, we obtain a covering of $R_1$ by
segments $S_r=[|y_r|-\ep_r,|y_r|+\ep_r]$. Extract a disjoint subfamily $(S_{r_j})$ such that
$2S_{r_j}$ covers $R_1$. Now, resend back this subfamily onto their original places, i.e. consider only the subset of $\ds \bigcup_{r\in[0,1]}\{B_{\ep_r}(y_r)\}$ whose translations are in the disjoint subfamily $(S_{r_j})$. With this subfamily and by using their radii we obtain the mutually disjoint partition of $(0,1)$ mentioned at the beginning of this Section \ref{SEctionNOnDeg}. Next, we show that the total mass of
$\Delta u$ in $B_{\ep_{r_j}}(y_{r_j})$ is at least  $c\ep_{r_j}$ so that when we add them up we get $1/2$ (it is because $2S_{r_j}$ covers $R_1$). In this part of the proof, we will use the linear growth property of $u$ in $B_{\ep_{r_j}}(y_{r_j})$ which is true by Theorem \ref{linearGrowth}. Let $w$ be the harmonic function such that
\begin{equation*}
\begin{cases}
 \triangle w =0 &\text{ in}\; B_{\ep_{r_j}}(y_{r_j}), \\
w= u &\text{ on }\; \partial B_{\ep_{r_j}}(y_{r_j}).\\
\end{cases}
\end{equation*}

By the Divergence theorem, we have
$$
\begin{array}{lcl}
   0&=& \ds\int_{\partial B_{\ep_{r_j}}(y_{r_j})}(u-w) (u-w)_\nu \\& =&  \ds\int_{ B_{\ep_{r_j}}(y_{r_j})}\Delta(u-w)(u-w)+ \nabla^2(u-w)dx \\
&\geq & \ds\int_{ B_{\ep_{r_j}}(y_{r_j})}\Delta(u-w)(u-w)dx + \ds\int_{ B_{\ep_{r_j}/2}(z_{r_j})}\nabla^2(u-w)dx \\
 &\geq & \ds\int_{ B_{\ep_{r_j}}(y_{r_j})}\Delta(u-w)(u-w)dx+ \frac{4C}{\ep^2_{r_j}}\ds\int_{ B_{\ep_{r_j}/2}(z_{r_j})}(u-w)^2 dx \\
 &\geq&  \ds\int_{ B_{\ep_{r_j}}(y_{r_j})}\Delta(u-w)(u-w)dx+ \frac{C_2}{\ep^2_{r_j}}\ds\int_{ B_{\ep_{r_j}/2}(z_{r_j})}(\frac{\ep_{r_j}}{2})^2 dx \\
 &\geq&  \ds\int_{ B_{\ep_{r_j}}(y_{r_j})}\Delta(u-w)(u-w)dx+ C_3 \ep^2_{r_j}, \end{array}
$$
where $z_{r_j} \in \Omega^c(u)$ such that $u\equiv0$ and $w$ grows linearly in $B_{\ep_{r_j}/2}(z_{r_j})$.
Since, $0 < w-u \leq C_1\ep_{r_j} $ and $w$ is harmonic in $ B_{\ep_{r_j}}(y_{r_j})$, we get
$$ \begin{array}{lcl}
  C_4 \ds\int_{ B_{\ep_{r_j}}(y_{r_j})}\ep_{r_j}\triangle u dx  &\geq & -\ds\int_{ B_{\ep_{r_j}}(y_{r_j})}(\Delta(u-w)(u-w))dx   \geq   C_3 \ep^2_{r_j} \\
    \ds\int_{ B_{\ep_{r_j}}(y_{r_j})}(\Delta u)dx &\geq &C_5 \ep_{r_j}.
   \end{array}
$$

Thus, by
adding these $\Delta u$ masses in these balls, $B_{\ep_{r_j}}(y_{r_j})$,
over $r_j$, we get the total mass of $\Delta u$ in
$B_1(0)$ is at least $C_5>0$, i.e.
$$C \geq\ds\int_{ B_1(0)}(\Delta u)dx  \geq \ds\sum_{r_j} \ds\int_{ B_{\ep_{r_j}}(y_{r_j})}(\Delta u)dx\geq C_5.$$
\end{proof}

\begin{thm}\label{NonDegrteinB} $u$ is Non-Degenerate \index{Non-Degenerate @\emph{Non-Degenerate}}, i.e. there exists a universal constant
$\kappa>0$  such that $\ds \sup_{B_r(x_0)} u(x)\geq
\kappa r$, for every $x_0 \in \partial\Omega(u).$
\end{thm}
\begin{proof}(By way of contradiction ) Assume that there exists $x_0 \in \partial\Omega(u)$ with $r>0$ such that
$\ds \sup_{B_r(x_0)} u(x)<
\delta r,$
for some sufficiently small $\delta>0$. By Green's representation theorem,
we have $$
\begin{array}{ccc}
 0= u(x_0) & = & \ds\int_{\partial B_{r}(x_0)}u G_\nu ds + \ds\int_{
B_{r}(x_0)}(G  \Delta u)dx,
\end{array}
$$
where $G(y,x_0)= \ds\frac{1}{2\pi} \ln \ds(\frac{1}{r}|y-x_0|)$ with $
G \equiv 0$ on $\partial B_{r}(x_0)$. Hence, $$\ds\int_{\partial
B_{r}(x_0)}u G_\nu ds =- \ds\int_{ B_{r}(x_0)}(G  \Delta u)dx \geq - \ds\int_{ B_{r/2}(x_0)}G  \Delta u dx. $$

Then, by Lemma \ref{LemLapRinB} and $G(y,x_0) \leq -C_1$ in $B_{r/2}(x_0)$, we get
$$
\begin{array}{ccl}
  \ds\int_{\partial B_{r}(x_0)}u G_\nu ds &\geq & -\ds\int_{ B_{r/2}(x_0)}(G
\Delta u)dx\\
  &\geq& C_2 r, \; \text{for some universal constant} \; C_2>0.
\end{array}
$$

$G_\nu \sim \ds\frac{1}{r}$ on $\partial B_r(x_0)$ so we obtain by the assumption
$$
\begin{array}{ccl}
 C_3 \delta r &\geq & \ds\int_{\partial B_{r}(x_0)}u G_\nu ds\geq  C_2 r.
\end{array}
$$
Contradiction; we can choose $\delta>0$ small enough so that the
above inequality fails. Hence, we get the result.
\end{proof}

\section{Locally, $\Omega(u)$
has a Single Component with
a Positively Dense Complement}\label{ChapLocalResults}
 Let us consider a point $x_0 \in \partial\Omega(u)$ and a neighborhood of $x_0$, $B_r(x_0),$ such that all the components of $\Omega(u)$ reach up to at least $B_{r/2}(x_0)$ and all the components of $\{x|u(x)=0\}$ reach $0$ with a connected subset of $\partial\Omega(u)$. Then, we normalize this neighborhood to $B_1(0)$ by taking $\tilde{u}(x)= \ds \frac{u(x_0-rx)}{r}$ for $x\in B_1(0)$. For the sake of simplicity, we denote the normalized function $\tilde{u}(x)$ as $u(x)$. We will show that $\Omega(u)$
has a single component with
a positively dense complement in $B_1(0)$ with the following steps:
\begin{enumerate}
 \item Let $\Upsilon$ be any connected component of $\{x|u(x)=0\}$, then $\Upsilon$ has some nice geometric properties.
  \item Let $\Gamma$ be any connected component of $\Omega(u)$ in $B_1(0)$, then $u$ has a nontrivial linear growth in $\Gamma$, i.e. $u$ has Non-Degeneracy component by component.
  \item $\Omega(u)$ has at most two connected
components.
  \item There is only one component with a positively dense complement.
\end{enumerate}
\subsection{On Some Properties of the Open Components of Zero-Level Set}
 Main results of this section are the following:
  \begin{itemize}
   \item There is no open component of $\{x|
u(x)=0\}^o$ which is strictly contained in $B_1(0)$
    \item Let $\Upsilon$ be any open component of $\{x|
u(x)=0\}^o$ with $r_0= d(\Upsilon,0)$, then the contribution of the mass of $\Delta u$ in $\Upsilon$ $\sim (1-r_0)$ in $B_1(0)$,
    \item Let $\Upsilon$ be connected to $0$ with the connected subset, $\mathcal{C}_\Upsilon$, of $\partial \Omega(u)$, then $\partial\Upsilon \cap \partial B_1(0)\neq \emptyset$ with one of the following two conclusions: either
\begin{enumerate}
  \item for any $\eta, \ep>0$ and $x \in \mathcal{C}_\Upsilon\cap B_{1-\eta}(0)$ we have $B_\ep(x)\cap \{x| u(x)=0\}^o \neq \emptyset$
  and the set $$\ds \bigcup_{x\in\mathcal{C}_\Upsilon\cap B_{1-\eta}(0)}\{O |\; O \;\text{is a component of }\;\{x|u(x)=0\}^o \;\text{such that}\; B_\ep(x)\cap O \neq \emptyset \}$$ has finitely many elements, or
  \item $0\in \partial\Upsilon$.
\end{enumerate}
  \end{itemize}
  Next two lemmas provide us some lower bound estimates on the rate of growth of some harmonic functions:

   \begin{lem}\label{unbddgrad}Let $h(x)=\langle x, e_2 \rangle_+= x_2^+$ in $ \mathbb{R}^2$ and $w$ be a harmonic function such that $w \geq h$ in $B_1(0)$ and $w=0$ on $\partial B_1(0) \cap \{x| \langle x,e_2\rangle <0\}$. Then, for $\nu$ is the inner normal vector to $\Omega(w)$ at $x_0\in \partial\Omega(w)
\cap \partial B_1(0) \cap \{x| \delta < \langle x,e_2\rangle <0\}$ we have  $ w_\nu(x_0) \geq \ln(\delta) $ for some $C>0$.
\end{lem}
\begin{proof} We can write down the Poisson formula for $w$ and estimate
$w_\nu(x_0)$ by the first order incremental quotient, so we have
 $x \in B_1(0)$,

$$w(x)=\ds\frac{1-|x|^2}{2} \int_{\partial
B_1(0)}\frac{w(y)}{|x-y|^2}dS(y).$$
Let us denote $\langle x_0+s\nu,e_1\rangle = a$ and $\langle x_0+s\nu,e_2\rangle = b \leq 0$, then we have
$$\begin{array}{lcl}
   \ds  \frac{w(x_0+ s\nu)-w(x_0)}{s}& = & \ds\frac{1-|x_0+s \nu|^2}{2 s}
\int_{\partial B_1(0)}\frac{w(y)}{|x_0+s\nu-y|^2}dS(y)\\
     & \geq & \ds\frac{[1-|x_0+s\nu|^2]}{2 s} \int_{\partial
B_1(0)}\frac{h(y)}{|x_0+s\nu-y|^2}dS(y)
\\
     & =& \ds\frac{[1-|x_0+s\nu|^2]}{2 s} \int_{\partial
B_1(0)\cap \{y_2>0\}}\frac{y_2}{|x_0+s\nu-y|^2}dS(y)
   \end{array}
$$
Let us take the limit as $s \to 0$ on both sides, then
for $x_0\in \partial\Omega(w)
\cap \partial B_1(0) \cap \{x| \delta < \langle x,e_2\rangle <0\}$ we have
$$\begin{array}{lcl}
  w_\nu(x_0) &\geq &\ds C\int_{\partial
B_1(0)\cap \{y_2>0\}}\frac{y_2}{|x_0-y|^2}dS(y) \\ &\geq& \ds C\int_{\partial
B_1(0)\cap \{y_2>0\}\cap \{|x_0-y| \rangle< 2\delta \}}\frac{y_2}{|x_0-y|^2}dS(y)
\\ &\geq& \ds C\int_{\partial
B_1(0)\cap \{y_2>0\}\cap \{|x_0-y|< 2\delta \}}\frac{y_2}{2-2\delta^2 -2y_1x_1}dS(y)
\\ &=&C \ln[\ds\frac{ 2-2\delta^2-2\sqrt{1-\delta^2}x_1}{2-2\delta^2-2x_1}] \geq C \ln(\delta),
\end{array}$$
since $x_1 \geq \sqrt{1-\delta^2}$ where $x_1=\langle x_0,e_1\rangle$.
Hence, we get the result.
\end{proof}
\begin{lem}\label{biggradient for2planes}Let $h(x)=(\langle x, \pm e_2 \rangle \pm \delta)_+$ in $ \mathbb{R}^2$ and $w$ be a harmonic function such that $w \geq h$ in $B_1(0)$. Then, for $\nu$ is the inner normal vector to $\Omega(w)$ at $x_0\in \partial\Omega(w)$ we have $ w_\nu(x_0) \geq \frac{C}{\delta}$ for some $C>0$.
\end{lem}
\begin{proof} Notice that $\{x|w(x)=0\} \subseteq \partial B_1(0)$ and we can write down the Poisson formula for $w$ and estimate
$w_\nu(x_0)$ by the first order incremental quotient, as we did in the proof of Lemma \ref{unbddgrad}, so we have

$$\begin{array}{ccc}
   \ds  \frac{w(x_0+h\nu)-w(x_0)}{h}& = & \ds\frac{1-|x_0+h\nu|^2}{2\pi h}
\int_{\partial B_1(0)}\frac{w(y)}{|x_0+h\nu-y|^2}dS(y)\\
     & \geq & \ds\frac{[1-|x_0+h\nu|^2]}{2\pi h} \int_{\partial
B_1(0)}\frac{h(y)}{|x_0+h\nu-y|^2}dS(y)\\
     &\geq &
 \ds\frac{[1-|x_0+h\nu|^2]}{2\pi h} \int_{\partial
B_1(0)\cap \{|x_0-y|< \delta\}}\frac{h(y)}{\delta^2}dS(y)\end{array}
$$
Let us take the limit as $h \to 0$ on both sides, we get
$$w_\nu(x_0) \geq C  \int_{\partial
B_1(0)\cap \{|x_0-y|< \delta\}}\frac{1}{\delta^2}dS(y)\geq \frac{C}{\delta}.$$

\end{proof}

\begin{lem}\label{nointerior} There is no open component of $\{x|
u(x)=0\}^o$ which is strictly contained in $B_1(0)$.
\end{lem}
\begin{proof}
(By way of contradiction) Assume there is an open component of $\{x| u(x)=0\}^o$, say $\Gamma$ which is
strictly contained in $B_1(0)$, then there is a tangent ball
$B_r(y)\subseteq\{x| u(x)=0\}^o$ to $\Omega(u)$, let $x_0\in
\partial\Omega(u)\cap \partial B_r(y)$. Then, by Lemma \ref{linearGrowth}, $x_0$
has a ball $B_r(x_0)$ such that $u$ has a linear growth in
$B_r(x_0)$. Consider the domain $\Sigma= B_r(x_0) \cup \Gamma$, shown in Figure \ref{ResmNonempty}, and
$h(x)$ be the harmonic function such that
\begin{equation*}
\begin{cases}
 \triangle h =0 &\text{ in}\; \Sigma, \\
h= u &\text{ in } \;B_1(0)\setminus \Sigma.\\
\end{cases}
\end{equation*}
\begin{figure}
  \includegraphics[width=2.5in]{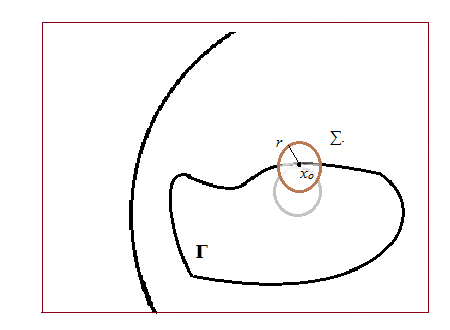}\\
  \caption{$\Sigma= B_r(x_0) \cup \Gamma$}\label{ResmNonempty}
\end{figure}

Then, $h \geq u$ in $\overline{\Sigma}$. We claim that $v=\max\{u,h\}$ is a larger subsolution than $u$, we know that $h$ is a harmonic function with $h \neq u$, so $v$ is a subharmonic function in $\Omega(v)$ so it is enough to show that $v$ satisfies the FBC, $(iii)$ in Definition \ref{Def1}: Let $y_0 \in \partial\Omega(v)$ such that $y_0$ has a tangent ball from outside of $\Omega(v)$ and $\eta$ be the inner normal vector into $\Omega(v)$, then $y_0 \in \partial  \Sigma$ by the Maximum principle. Moreover, the only possible region for $y_0$ is either in a neighborhood of the intersection points $\partial B_r(x_0) \cap \partial \Gamma $ (since outside of these neighborhoods zero-level set of $h(x)$ can be only a curve) or $y_0 \in \partial\Omega(u)$ such that $y_0$ has a tangent ball from outside of $\Omega(u)$.\par In the first case, we estimate $v_\eta(y_0)$ by Lemma \ref{unbddgrad}. $v$ is a harmonic function which is bigger than $l(x)=\alpha\langle x-x_0, \nu \rangle_+$ in $B_r(x_0)$ for some direction $\nu \in S^1$ and $\alpha>0$ since $u$ grows linearly in $B_r(x_0)$.
Then, we have
$$v_\eta(y_0) \geq C \ln(|y_0-x_0|)>0, $$
Hence, $h_\eta(y_0) \geq C \ln(r) $. Therefore, by choosing $r$ small enough we guarantee that $v$
is a subsolution of \eqref{eq:FB}. \par In the second case, if we have $v(x)\leq \alpha\langle x-y_0, \nu \rangle_+ + o(|x-y_0|)$ for some $\al>0$, then we have $\al \geq \sqrt{f(y_0)}$ since $y_0 \in \partial\Omega(u)$  with a tangent ball from outside of $\Omega(u)$, $$u(x)\leq v(x) \leq \alpha\langle x-y_0, \nu \rangle_+ + o(|x-y_0|),$$
and $u$ satisfies the FBC, $(iii)$ in Definition \ref{Def1}, i.e. $\al \geq \sqrt{f(y_0)}$. Hence, $v$ is a  subsolution of \eqref{eq:FB}.
\par Thus, we construct a larger subsolution than $u$, contradiction. Hence, the result follows.

\end{proof}

Lemma \ref{nointerior} is trivially true for $\Omega(u)$ by the Maximum principle. Since, the harmonic function which is  zero on the boundary is the identically zero function.

\begin{lem}\label{LaplaceinZeroSet}
   Let $\Upsilon$ be any open component of $\{x|
u(x)=0\}^o$ with $r_0= d(\Upsilon,0)$, then the contribution of the mass of $\Delta u$ in $\Upsilon$ $\sim (1-r_0)$ in $B_1(0)$
\end{lem}
\begin{proof}
The upper bound is trivial, for the lower bound we will use the linear growth in a ball which has a tangent ball from the zero level set.
By Lemma \ref{linearGrowth}, for any $r\in (r_0,1)$, there exists $y_0 \in \{x|
u(x)=0\}^o$ and $\ep_{y_0}>0$ such that $B_{\ep_{y_0}}(y_0)$ is tangent to $\Omega(u)$, say at $x_0$
and $u$ has a linear growth in $B_{\ep_{y_0}}(x_0)$.  Pick a ray in $B_1(0)$, say  $R_1$. Now, we use $R_1$ in order to pick a partition of $[r_0,1]$ in terms of $\ep_{y_0}$. Consider $\{B_{\ep_{y_0}}(x_{\ep_{y_0}})| r\in[r_0,1]\}$ and project all these balls
$B_{\ep_{y_0}}(x_{\ep_{y_0}})$ onto $R_1$ . Hence, we obtain a covering of $R_1$ by
segments $S_r=[r-h,r+h]$. Extract a disjoint subfamily $(S_{r_j})$ such that
$2S_{r_j}$ covers $R_1$. Now, resend back this subfamily onto their original places. Next, we know that the total mass of
$\Delta u$ in $B_{\ep_{r_j}}(x_{r_j}) \sim \ep_{r_j}$, by Lemma \ref{LemLapRinB}, so that when we add them up we get a lower bound as $c(1-r_0)$ for some $c>0$.
\end{proof}
Lemma \ref{LaplaceinZeroSet} is a variation of Lemma \ref{LemLapRinB}.
\begin{lem}\label{lemmazerointerior} Let $\Upsilon$ be any connected component of $\{x| u(x)=0\}^o$ in $B_1(0)$ which is connected to $0$ with the connected subset, $\mathcal{C}_\Upsilon$, of $\partial \Omega(u)$. Then $\partial\Upsilon \cap \partial B_1(0)\neq \emptyset$ with one of the following two conclusions: either
\begin{enumerate}
  \item for any $\eta, \ep>0$ and $x \in \mathcal{C}_\Upsilon\cap B_{1-\eta}(0)$ we have $B_\ep(x)\cap \{x| u(x)=0\}^o \neq \emptyset$
  and the set $$\ds \bigcup_{x\in\mathcal{C}_\Upsilon\cap B_{1-\eta}(0)}\{O |\; O \;\text{is a component of }\;\{x|u(x)=0\}^o \;\text{such that}\; B_\ep(x)\cap O \neq \emptyset \}$$ has finitely many elements, or
  \item $0\in \partial\Upsilon$.
\end{enumerate}
See Figure \ref{ResimOPenCom}.
\begin{figure}
  \includegraphics[width=2in]{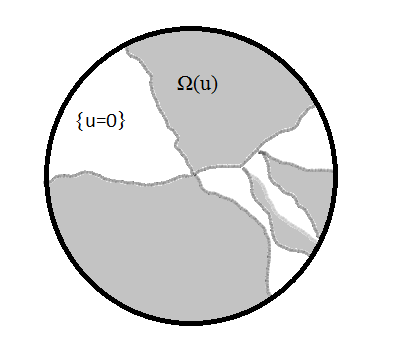}\\
  \caption{Possible configuration of the components in $B_1(0)$}\label{ResimOPenCom}
\end{figure}

\end{lem}

\begin{proof}(By way of contradiction) Let $\Upsilon$ is a connected component of $\{x| u(x)=0\}^o$ in $B_1(0)$. By Lemma \ref{nointerior}, $\Upsilon$ cannot be strictly inside of $B_1(0)$, so the only possibility is that $\partial\Upsilon \cap \partial B_1(0)\neq \emptyset$, $0 \notin \partial\Upsilon$, and either there exists $x_0\in \mathcal{C}_\Upsilon\cap B_{1-\eta}(0)$, for some $\eta>0$, with a ball $B_r(x_0)$ such that $B_r(x_0)\cap \{x| u(x)=0\}^o = \emptyset $ or the set
 $$\ds \bigcup_{x\in \mathcal{C}_\Upsilon\cap B_{1-\eta}(0)}\{O |\; O \text{is a component of }\;\{x|u(x)=0\}^o \;\text{such that}\; B_{\ep_0}(x)\cap O \neq \emptyset \}$$
 has infinitely many elements for some $\ep_0>0$. \par If there exists a ball $B_r(x_0)$ such that $B_r(x_0)\cap \{x| u(x)=0\}^o = \emptyset $, then consider the harmonic function $h(x)$ defined in $B_{r/2}(x_0)$ with $h=u$ on $\partial B_{r/2}(x_0)$. Hence, $h(x)>u(x)$ in $B_{r/2}(x_0)$ moreover $\max\{u,h\}$ is a larger subsolution than $u$. Contradiction.
 \par If the set
 $$\ds \bigcup_{x\in\mathcal{C}_\Upsilon\cap B_{1-\eta}(0)} \{O | \; O\; \text{is a component of } \;\{x|u(x)=0\}^o \;\text{such that} \; B_{\ep_0}(x)\cap O \neq \emptyset \}$$
 has infinitely many elements for some $\ep_0>0$, then every component has a contribution to the total mass of $\Delta u$ by Lemma \ref{LaplaceinZeroSet} which is at least $\min\{\ep_0,\eta\}$. On the other hand, the total mass of $\Delta u$ is finite in $B_1(0)$. Contradiction. Thus, there are at most finitely many distinct connected components of $\{x|u(x)=0\}^o$ in this case. Hence, the result follows.
\end{proof}
\subsection{$u$ has a nontrivial linear growth in $\Gamma$}

\begin{lem}\label{finiteHausdrfMeas.} Let $\Gamma$ be any connected component of $\Omega(u)$ in $B_1(0)$, then $\mathcal{H}^1(\partial\Gamma) < +\infty$.
  \end{lem}
\begin{proof}
Let us restrict $u$ only on $\Gamma$ and work on the total mass of $\Delta u$ in $\Gamma$. Let us denote $w=u|_{\Gamma}$ in $B_1(0)$ with $w\equiv 0$ in $B_1(0)\backslash\Gamma$. By Lemma \ref{LemLapRinB}, there exists a universal constant $C>0$ such that we have
$$ C \geq  \ds\int_{B_1(0)}\Delta u dx \geq  \ds\int_{\Gamma}\Delta w dx.
$$

Since $\Gamma$ can be covered by a countable union of almost disjoint balls of radius $\ep$ and $w$ is harmonic in $\Gamma$, we have
$$
\begin{array}{c}
  C  \geq \ds\int_{\Gamma}\Delta w dx = \ds \sum_{j} \ds\int_{B_\ep(x_j)}\Delta w dx.
\end{array}
$$
If $B_\ep(x_j) \subseteq \Gamma$, then $w$ is harmonic in $B_\ep(x_j)$; so $\ds\int_{B_\ep(x_j)\subseteq \Gamma}\Delta w dx=0$. If $B_\ep(x_j) \cap \partial\Gamma \neq\emptyset$, then the total mass of $\triangle w \sim \ep$ in $B_\ep(x_j)$, by Lemma \ref{LemLapRinB}; so there exists $\kappa>0$ such that $\ds\int_{B_\ep(x_j)}\Delta w dx \geq \kappa \ep$.
Thus,

$$
\begin{array}{lcl}
  C & \geq &\ds \sum_{j} \ds\int_{B_\ep(x_j)}\Delta w dx\\
  & \geq &\ds \sum_{\{j|B_\ep(x_j) \cap \partial\Gamma \neq \emptyset \}} \ds\int_{B_\ep(x_j)}\Delta w dx \\
  & \geq &\ds \sum_{\{j|B_\ep(x_j) \cap \partial\Gamma \neq \emptyset \}} \kappa \ep. \\
\end{array}
$$
Hence, the number of balls with radius $\ep$ that cover $\partial\Gamma$ is at most $\frac{C}{\kappa\ep}$.
As a result, we obtain that $\mathcal{H}^1(\partial\Gamma) < +\infty$.
\end{proof}


\begin{lem}\label{zerosetlinearly}
  Let $w$ be a harmonic function in $\Omega \subseteq B_1(0)$ with the following properties:
  \begin{enumerate}
    \item $0\in \partial\Omega$,
    \item $\Omega$ is in the upper half-plane, i.e. $\Omega \subseteq B_1(0)\cap \{x| x_n \geq 0\}$,
    \item $w$ has a linear growth in $\Omega$,
  \end{enumerate}
then for any cone, $C_0$, in $B_1(0)\cap \{x| x_n \geq 0\}$ and for any $r\in [0,1]$, the number of $n$ that satisfies $r_n=\ds\frac{r}{2^n}$  with $\{x\in C_0| |x|< {r_n}/2, w(x)=0\}\neq \emptyset$ is finite. Shown in Figure \ref{ResimCone}.
\begin{figure}
\includegraphics[width=2in]{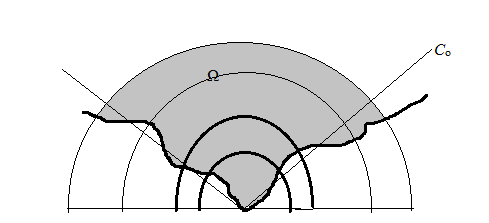}\\
  \caption{Number of $r_n$ is finite}\label{ResimCone}
\end{figure}

\end{lem}
Note that $C_0$ has the following representation:
$$C_0= \{(y,x_n)\in B_1(0)|x_n \geq0 \;\text{and}\; |y|\leq 1-a\}\;\text{for some} \;a>0.$$
\begin{proof}
  (By way of contradiction) Assume that there exist a cone $C_0$ and $r>0$ with a sequence of the form  $r_{n_k}=\ds\frac{r}{2^{n_k}}$ such that $\{x\in C_0| |x|< \frac{r_{n_k}}{2}\; \text{and}\; w(x)=0\}\neq \emptyset$. For simplicity, let us represent this sequence by $r_k$, then there exists $y_k$ such that $|y_k|\leq {r_k}/2$ and $w(y_k)=0$. Then, by Lipschitz property, $w(x)\leq \kappa\langle x, e_n\rangle$ in $B_1(0)\cap \{x| x_n \geq 0\}$ for the Lipschitz constant $\kappa>0$ of $w$. Moreover, $\kappa\langle x, e_n\rangle-w(x)$ is a harmonic function in $B_1(0)\cap \{x| x_n \geq 0\}$ so by the Harnack inequality we have  $ \kappa\langle y_k, e_n\rangle \leq C(\kappa\langle x, e_n\rangle-w(x) )$ in $B_{{r_k}/2}(0)$. Therefore,
  $$ w(x) \leq \kappa\langle x, e_n\rangle-\frac{\kappa}{C} \langle y_k, e_n\rangle \;\text{in}\;B_{{r_k}/2}(0).$$
  By construction, we have $B_{r_{k+1}}(0)\subseteq B_{{r_k}/2}(0)$ and $y_{k+1}\in B_{r_{k+1}/2}(0)$
with $w(y_{k+1})=0$. Consider the harmonic function   $$\kappa\langle x, e_n\rangle-\frac{\kappa}{C} \langle y_k, e_n\rangle - w(x)\geq 0\;\text{in}\;B_{{r_k}/2}(0);$$
by the Harnack inequality, we have
$$C\kappa\langle x, e_n\rangle-\kappa \langle y_k, e_n\rangle -C w(x) \geq \kappa\langle y_{k+1}, e_n\rangle-\frac{\kappa}{C} \langle y_k, e_n\rangle \;\text{in}\;B_{{r_{k+1}}/2}(0).$$
Hence,
$$w(x) \leq \kappa\langle x, e_n\rangle-\frac{\kappa}{C} \langle y_k, e_n\rangle- \frac{\kappa}{C}\langle y_{k+1}, e_n\rangle+\frac{\kappa}{C^2} \langle y_k, e_n\rangle  \;\text{in}\;B_{{r_{k+1}}/2}(0).$$
If we continue this iteration, we obtain a decay on the right hand side faster than linearity. This contradicts to the linear growth of $w$. Hence, the result follows.

\end{proof}
So far, we know that for any $x \in \partial\Omega(u)$ and a given ball $B_r(x)$, Non-Degeneracy condition will be attained from a connected component of $\Omega(u)$ in $B_r(x)$ but not necessarily will be attained from all the components of $\Omega(u)$. Next, we will show that Non-Degeneracy condition is true for each of the connected component of $\Omega(u)$ in $B_r(x)$.
\begin{lem}\label{NonDg_in_Gamma}
  Let $\Gamma$ be any connected component of $\Omega(u)$ in $B_1(0)$ and $B_r(y)$ is tangent from inside to $\partial\Gamma$ at $x_0$, then there exists a universal constant $C>0$ such that
   $$\ds\sup_{B_h(x_0)\cap \Gamma} u \geq C h, \; \text{for any } h \sim r.$$
\end{lem}
\begin{proof}
  (By way of contradiction) Assume that for sufficiently small $\delta>0$, we can find $r_n>0$ such that
  $$\ds\sup_{B_{r_n}(x_0)\cap \Gamma} u \leq \delta r_n.$$
  By Theorem \ref{NonDegrteinB}, $u$ is Non-Degenerate so there exists $C>0$ such that
   $$\ds\sup_{B_{r_n}(x_0)} u \geq C r_n.$$
   Therefore, for $\delta>0$ small enough, $\Omega(u)\setminus \Gamma$ is nonempty around $x_0$ and$u$ grows linearly in this set. Let us denote this set as $\Gamma_o$ and we have $\Gamma_o \subseteq \Gamma^c$. Let $\eta$ be the inner normal vector of $\Gamma_o$ at $x_0$ then we can normalize $B_{r_n}(x_0)$ to $B_1(0)$ with mapping $x_0 \mapsto 0$ and $\eta\mapsto e_2$ with $u(x)= e^{i\theta}\frac{w(x-x_0)}{r_n}$ where $\theta \geq 0$ is the angle in between $\eta$ and $e_2$. By Lemma \ref{zerosetlinearly}, for any cone $C_0 \subseteq B_1(0)\cap \{x| x_n \geq 0\}$ and for any $r\in [0,1]$, the number of $k$ that satisfies $r_k=\ds\frac{r}{2^k}$  with $\{x\in C_0| |x|< {r_k}/2, w(x)=0\}\neq \emptyset$ is finite. Hence, for $\ep>0$ there exists $r_n>0$ small enough such that the arc-length of $\partial B_{r_n}(x_0)\cap B_r(y) > \pi r_n-\ep$ and
   the arc-length of  $\partial B_{r_n}(x_0)\cap \Gamma_o > \pi r_n-\ep$. Now, we can construct a larger subsolution by taking the harmonic function $h(x)$ in $B_{r_n}(x_0)$ such that $h(x)=u(x)$ for $x\in \partial B_{r_n}(x_0)$. We can estimate $h_\nu(x)$ for $x\in \partial\Omega(h)$ as we did in the proofs of Lemma \ref{unbddgrad} and \ref{biggradient for2planes}. We can write down the Poisson formula for $h$ and estimate
$h_\nu(x)$ by the first order incremental quotient so we have

$$\begin{array}{ccc}
   \ds  \frac{h(x+s\nu)-h(x)}{s}& = & \ds\frac{1-|x+s\nu|^2}{2\pi r_n s}
\int_{\partial B_{r_n}(x_0)}\frac{h(y)}{|x+s\nu-y|^2}dS(y)\\
     & \geq & \ds\frac{[1-|x+s\nu|^2]}{2\pi r_ns} \int_{\partial
B_{r_n}(x_0)}\frac{h(y)}{|x+s\nu-y|^2}dS(y)\\
     &\geq &
 \ds\frac{[1-|x+s\nu|^2]}{2\pi r_n s} \int_{\partial
B_{r_n}(x_0)\cap \{|x-y|< 4\ep\}}\frac{h(y)}{32\ep^2+ 2s^2}dS(y)\end{array}
$$
Let us take the limit as $s \to 0$ on both sides, we get
$$h_\nu(x) \geq C  \int_{\partial
B_{r_n}(x_0)\cap \{|x_0-y|< 4\ep\}}\frac{1}{32\ep^2}dS(y)\geq \frac{C r_n}{4\ep}.$$
Hence, for sufficiently small $\ep>0$, we obtain $|\nabla h(x)|^2>\Lambda$ which implies that
$\max\{u,h\}$ is a larger subsolution than $u$. Contradiction. Hence, the result follows.
\end{proof}

\begin{thm}\label{nontrivialgrowth}$u$ has a nontrivial linear
growth in $\Gamma$, i.e. there exist universal constants
$C,c>0$ such that for any $x_0 \in \partial\Gamma$ and any $r \leq diam(\Gamma)$.

   \begin{equation}\label{nondeginCom}
   C r \geq\ds\sup_{B_r(x_0)\cap \Gamma} u \geq c r.
\end{equation}

\end{thm}
\begin{proof} First inequality is the direct result of Lipschitz property, so we need to prove the second inequality of \eqref{nondeginCom}. By Lemma \ref{nointerior}, we know that $\Gamma$ is a simple
connected domain. Let us just
consider $w=u|_\Gamma$ in $B_1(0)$ and denote $d_0=d(0,\Gamma)$. We will prove this theorem in two steps by combining and adapting the ideas of the proofs of Lemma \ref{LemLapRinB} and Theorem \ref{NonDegrteinB}. First, we will show that for any $r \in [d_0,1]$, there exists a ball where $w$ has a nontrivial growth. Second, we will obtain the inequality by way of contradiction with Green's representation theorem.
\par Let us start the proof of the first claim: $\Gamma$ has a curve from $d_0$ to $\partial B_1(0)$ and for any $r\in [d_0,1]$, there exists $x_r \in \Gamma$ and $B_{\ep_r}(x_r)\in \Gamma$ which is tangent from inside to $\Gamma$ at some point $y_r \in \partial\Gamma $. As we did before, pick a ray in $B_1(0)$, say  $R_1=\{r \eta | r\in[0,1], \eta \in S^1\}$. Now, we use $R_1$ in order to pick a partition of $[d_0,1]$ in terms of $\ep_r$. Consider $\{B_{\ep_r}(y_r)| r\in[d_0,1]\}$ and project all these balls
$B_{\ep_r}(y_r)$ onto $R_1$ . Hence, we obtain a covering of $R_1 \cap [d_0,1]$ by
segments $S_r=[r-h,r+h]$. Extract a disjoint subfamily $\{S_{r_j}\}$ s.t.
$2S_{r_j}$ covers $R_1 \cap [d_0,1]$. Now, resend back this subfamily to their original places. Since, $B_{\ep_r}(x_r)$ is tangent from inside to $y_r$, by Lemma \ref{NonDg_in_Gamma}, we have
$\ds\sup_{B_{\ep_r}(y_r)\cap \Gamma} u=  \ds\sup_{B_{\ep_r}(y_r)} w \geq c \ep_r.$ Now, as following the same steps of the proof of Lemma \ref{LemLapRinB}, we obtain
\begin{equation}\label{EqnLapinB}
\ds\int_{ B_{\ep_{r}}(y_{r})}(\Delta w)dx  \sim \ep_{r}.
\end{equation}

\par Now, let us prove the second inequality of \eqref{nondeginCom} by way of contradiction: Assume that there exists $x_0 \in \partial \Gamma$ with $r>0$ such that
$\ds \sup_{B_r(x_0)} w(x)<
\delta r,$
for some sufficiently small $\delta>0$. By Green's representation theorem,
we have $$
\begin{array}{ccc}
0=  w(x_0) & = & \ds\int_{\partial B_{r}(x_0)}w G_\nu ds + \ds\int_{
B_{r}(x_0)}(G  \Delta w)dx
\end{array}
$$
where $G(y,x_0)= \ds\frac{1}{2\pi} \ln \ds(\frac{1}{r}|y-x_0|)$ with $
G \equiv 0$ on $\partial B_{r}(x_0)$. Hence, $$\ds\int_{\partial
B_{r}(x_0)}w G_\nu ds =- \ds\int_{ B_{r}(x_0)}(G  \Delta w)dx \geq - \ds\int_{ B_{r/2}(x_0)}(G  \Delta w)dx. $$

Then, $G(y,x_0) \leq -C_1$ in $B_{r/2}(x_0)$ and by the first part of the proof we obtain a finite cover of $B_{r/2}(x_0) \cap \partial\Gamma$ with balls $\{B_{\ep_r}(y_r)| r\in[d_0,1]\}$.

$$
\begin{array}{ccl}
  \ds\int_{\partial B_{r}(x_0)}w G_\nu ds &\geq & -\ds\int_{ B_{r/2}(x_0)}(G
\Delta w)dx  \\
&=& -\ds\sum\int_{ B_{r/2}(x_0)\cap B_{\ep_r}(y_r) }(G
\Delta w)dx \\
  &\geq& C_2 r, \; \text{for some universal constant} \; C_2>0.
\end{array}
$$

The last inequality is true because we obtain $\Delta w \sim \ep_r $ in $B_{\ep_r}(y_r)$ by \eqref{EqnLapinB}.
$G_\nu \sim \ds\frac{1}{r}$ on $\partial B_r(x_0)$ so we obtain by the assumption
$$
\begin{array}{ccl}
 C_3 \delta r &\geq & \ds\int_{\partial B_{r}(x_0)}w G_\nu ds\geq  C_2 r.
\end{array}
$$
Contradiction; we can choose $\delta>0$ small enough so that the
above inequality fails. Hence, we get the result.
\end{proof}
\subsection{$\Omega(u)$ has a Single Component with a Positively Dense Complement}
So far, we obtain the nontrivial linear growth in every connected component of $\Omega(u)$ in $B_1(0)$, Theorem \ref{NonDg_in_Gamma}. Next, we will show that $\Omega(u)$ can have at most two components in $B_1(0)$. The intuitive
idea is the following: The component needs enough mass in $B_1(0)$ in order to have a nontrivial linear growth in $B_1(0)$. This idea directly connects this fact to the Monotonicity Formula. Because the Monotonicity formula enables us to find that how much mass a positivity set needs in order to have a specific growth-order.
\begin{rem}\label{RemLapAngle} Let $\Omega$ be a sector area enclosed by
the arc of length $\frac{2\pi}{\alpha}$ in $B_r(0)\subseteq \mathbb{R}^2$,
for some $\alpha \geq 1$, then $h(r,\theta)= r^{\alpha/2}
cos(\frac{\al}{2}(\theta+\frac{\pi}{\alpha}))$ is the harmonic function in
$\Omega$ with
\begin{equation*}
  \begin{cases}  \Delta h = 0 ,&\text{in}\;\; \Omega \\
    h(s,b)=h(s,\frac{2\pi}{\alpha})=0,& 0\leq s\leq r.
\end{cases}\end{equation*}
\end{rem}

 \begin{thm}\label{thmplanegrowth} For any given $\sigma >0$, there exist a pair of positive constants, $(\varepsilon,\delta)$, with $0<\ep<\delta<1$ such that if $\Omega(u) \cap B_1(0) \backslash B_\ep(0)$
has at least two components all of which intersects with $\partial  B_\ep(0)$, then, for some direction $e$, we have
$$ u(x) \geq C [\langle x,\pm e\rangle - \delta^2\sigma]_+, \;\text{in}\; B_\delta(0)\cap \{x|\langle x,\pm e\rangle - \delta^2\sigma>0\}.$$ Shown in Figure \ref{ResimPlaneBd}.
\begin{figure}
  \includegraphics[width=2.5in]{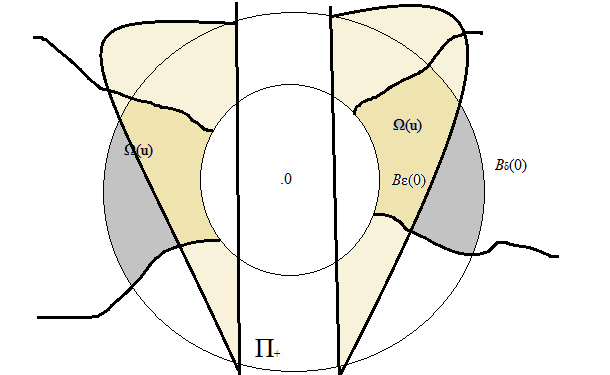}\\
  \caption{$u(x)$ is bigger than positive half planes, $\Pi_+$, in $B_\delta(0)$}\label{ResimPlaneBd}
\end{figure}

 \end{thm}
\begin{proof} (By way of contradiction) Assume that there exists a $\sigma >0$ such that, for any $(\varepsilon,\delta)$ pair with $\delta> \varepsilon >0$, we have $\Omega(u) \cap B_1(0) \backslash B_\ep(0)$
has at least two components and
$$ u(x) \leq C [\langle x,  e\rangle - \delta^2\sigma]_+ \; \text{in}\; B_\delta(0)\cap \{x|\langle x, e\rangle - \delta^2\sigma>0\} $$
or
$$ u(x) \leq C [\langle x, - e\rangle - \delta^2\sigma]_+ \; \text{in}\; B_\delta(0)\cap \{x|\langle x, - e\rangle - \delta^2\sigma>0\}, $$
for any direction $e$. Let us pick a sequence $\{\delta_k\}$ such that $\delta_k \ds\rightarrow 0$, as $k \rightarrow \infty$, a direction $e$, and $\ep_k=(1-\eta)\delta_k$, for some sufficiently small $\eta >0$, then we have
\begin{equation}\label{planeeqn}
u(x) \leq C [\langle x,\pm e\rangle - \delta_k^2\sigma]_+ \; \text{in}\; B_{\delta_k}(0)\cap \{x|\langle x,\pm e\rangle - \delta_k^2\sigma>0\},
\end{equation}
in at least one of the directions $e$ or $-e$ and $\Omega(u) \cap B_1(0) \backslash B_{\ep_k}(0)$ has at least two components that each of them intersects with $\partial B_{\ep_k}(0)$, without loss of generality, let us assume that there are two components and denote them as $\Omega_1, \Omega_2$. Since, $\Omega_1$ and $\Omega_2$ intersect with $\partial B_{\ep_k}(0)$, their diameter should be at least $\eta \delta_k$, by construction. Moreover, there exist $x_1 \in \partial \Omega_1 \cap B_{\ep_k}(0)$ and $x_2 \in \partial \Omega_2 \cap B_{\ep_k}(0)$. Therefore, by Theorem \ref{nontrivialgrowth},
\begin{eqnarray}\label{eq44}
                     \ds\sup_{\Omega_1 \cap B_{r}(x_1)} u \geq C r &\text{and}& \ds\sup_{\Omega_2 \cap B_{r}(x_2)} u \geq C r,
                          \end{eqnarray}
for any $r \leq \eta \delta_k\leq \min\{diam \Omega_1, diam \Omega_2\}$. We can choose $\delta_k, \eta>0$ small enough to contradict \eqref{planeeqn}. Hence, we obtain the result.
\end{proof}
\begin{cor}\label{onecomponent} There exist universal constants $h,\ep_0>0$ such that $\Omega(u)$ has only one component from $B_\ep(0)$ to $B_1(0)$ for any $\ep < \ep_0$. Moreover, the set $\{x|u(x)=0\}^o \cap B_{1/2}(0)$ contains a ball $B_h(z)$, for some $z\in B_{1/2}$.
\end{cor}
\begin{proof}[Proof of the first part](By way of contradiction)
Suppose that for $\ep=2^{-k_0}>0$ there exists a $\ep_0<\ep$ s.t. $B_1(0)\backslash B_{\ep_0}(0)$ has two components of $\Omega(u)$, say $\Omega_1$ and $\Omega_2$. These components are also components of $B_1(0)\backslash B_{\ep}(0)$. By Theorem \ref{NonDg_in_Gamma}, $u$ has a nontrivial growth in both $\Omega_1$ and $\Omega_2$. Let us consider $u_1(x)=u(x)|_{\Omega_1}$ and $u_2(x)=u(x)|_{\Omega_2}$, and write down the Monotonicity Formula, the Monotonicity Theorem \ref{MonoctyThm}, for them by denoting the universal Non-Degeneracy and Lipschitz constants as $c$ and $C$, respectively. Then, we have
 $$ c^4\pi^2 \leq J(\ep) \leq J(1) \leq C^4\pi^2.$$
 Note that we can obtain a lower bound for $J(r)$ by adapting the proof of Monotonicity Theorem \ref{MonoctyThm} as follows:
 $$
\begin{array}{lcl}
  \ds \frac{J'(r)}{J(r)}&=& \ds \frac{\int_{\partial B_r(0)}|\nabla u_1|^2d \sigma}{\int_{B_r(0)}|\nabla u_1|^2dx} + \ds \frac{\int_{\partial B_r(0)}|\nabla
u_2|^2d \sigma}{\int_{B_r(0)}|\nabla
u_2|^2dx} - \frac{4}{r}\\
&\geq& \ds \frac{\left(\int_{\partial B_r(0)}(u_1)_\theta^2d \sigma\right)^{1/2}}{\left(\int_{\partial B_r(0)} u_1^2d\sigma\right)^{1/2}} + \ds \frac{\left(\int_{\partial B_r(0)}
(u_2)^2_\theta d \sigma\right)^{1/2}}{\left(\int_{\partial B_r(0)}
u_2^2d \sigma\right)^{1/2}} - \frac{2}{r}.
\end{array}
$$
 If we denote the angular traces of the domains $\Omega_1$ and $\Omega_2$ in the circle of radius $r$ as $\pi t_1(r)$ and $\pi t_2(r)$, respectively, then
 the sum
 $$\ds \frac{\left(\int_{\partial B_r(0)}(u_1)_\theta^2d \sigma\right)^{1/2}}{\left(\int_{\partial B_r(0)} u_1^2d\sigma\right)^{1/2}} + \ds \frac{\left(\int_{\partial B_r(0)}
(u_2)^2_\theta d \sigma\right)^{1/2}}{\left(\int_{\partial B_r(0)}
u_2^2d \sigma\right)^{1/2}}$$
 attains its minimum for two adjacent, complementary arcs with length $\alpha 2\pi r$ and $(1-\alpha)2\pi r$ and the corresponding eigenfunctions are
$$\begin{array}{ccc}
    \sin\frac{\theta}{2\alpha r} & \text{and}& \sin\frac{\theta}{2(1-\alpha)r}.
  \end{array}
$$
Thus, we obtain

 \begin{equation}\label{EqnJ}
 \ds \frac{J'(r)}{J(r)} \geq \frac{1}{2r t_1(r)}+ \frac{1}{2r t_2(r)}- \frac{2}{r}
\end{equation}
Consider the right hand side of \eqref{EqnJ} as a function of $(t_1,t_2)$, i.e. let
$$F(t_1,t_2)= \frac{1}{2r t_1}+ \frac{1}{2r t_2}- \frac{2}{r},$$
then $F(t_1,t_2)$ has a minimum of zero at $t_1=t_2=1$ and it is strictly convex at $t_1=t_2=1$. Therefore,
$$
F(t_1,t_2) \geq \frac{c_0}{r}[(t_1-1)^2+(t_2-1)^2], \;\text{for some}\; c_0>0.
$$

If we consider $\tilde{u}(x)=[u(x)-c\ep]_+$, then $\tilde{u}$ has at least two components $\tilde{\Omega_1}\subseteq \Omega_1$ and $\tilde{\Omega_2}\subseteq\Omega_2$. Moreover, it will start to grow linearly from say $K \ep$, for some $K>0$. If we denote the Monotonicity Formula, in Monotonicity Theorem \ref{MonoctyThm}, for $\tilde{u}$ as $\tilde{J}(r)$ and the angular traces of the domains $\tilde{\Omega}_1$ and $\tilde{\Omega}_2$ in the circle of radius $r$ as $\pi \tilde{t}_1(r)$ and $\pi \tilde{t}_2(r)$, respectively,, then we have $\tilde{J}(r)\sim 1$, for any $r \in [K\ep,1]$. Therefore,
for $k_1>k_0>0$ and $2^{-k_1} > K\ep$, we have

$$
C_1 \geq \ds \int_{2^{-k_1}}^1 \frac{\tilde{J}'(r)}{\tilde{J}(r)} dr \geq \ds \int_{2^{-k_1}}^1 \frac{c_0}{r}[(\tilde{t}_1-1)^2+(\tilde{t}_2-1)^2] dr, $$

for some $C_1>0$.
Let us write down the right hand side integral with diadic representation:
$$
\begin{array}{ccc}
  C_1 & \geq &  \ds\int_{2^{-k_1}}^1 \frac{c_0}{r}[(\tilde{t}_1-1)^2+(\tilde{t}_2-1)^2] dr\\
&=& \ds \sum_{l=k_1}^{+\infty}  \ds \int_{2^{-l}}^{2^{-l+1}} \frac{c_0}{r}[(\tilde{t}_1-1)^2+(\tilde{t}_2-1)^2] dr\\
&\geq& \ds \sum_{l=k_1}^{+\infty} 2^l \ds \int_{2^{-l}}^{2^{-l+1}} c_0[(\tilde{t}_1-1)^2+(\tilde{t}_2-1)^2] dr.
\end{array}
$$
  For $\eta>0$ sufficiently small, there exists at least one ring such that
  $$
   C_1 \eta  \geq 2^l \ds \int_{2^{-l}}^{2^{-l+1}} c_0[(\tilde{t}_1-1)^2+(\tilde{t}_2-1)^2] dr,
  $$
  i.e. $[(\tilde{t}_1-1)^2+(\tilde{t}_2-1)^2] $ is close to zero most of the time therefore $\tilde{t}_1$ and $\tilde{t}_2$ are close to 1 most of the time for the radii between $2^{-l}$ and $2^{-l+1}$. This contradicts to Lipschitz property of a free boundary point $x_0 \in B_{r^*}(0)$ of $\tilde{u}$, where $r^*= \frac{1}{2}[2^{-l}+2^{-l+1}]$. There exists at least one free boundary point of $\tilde{u}$ for each radius such that this point has a neighborhood with only one component of $\Omega(\tilde{u})$, i.e. for each radius $r \in [2^{-l},r^*]$, $\tilde{t}_1$ and $\tilde{t}_2$ are different than $1$. That is because: let us suppose that $z_i \in \partial\tilde{\Omega}_i \cap \partial B_r(0)$ with $d(\tilde{\Omega}_1 \cap \partial B_r(0),\tilde{\Omega}_2\cap \partial B_r(0))=d(z_1,z_2)$, i.e. $d(z_1,z_2)$ gives the distance between $\tilde{\Omega}_1$ and $\tilde{\Omega}_2$ on $\partial B_r(0)$ so that we can determine whether $\tilde{t}_1$ and $\tilde{t}_2$ are different than $1$ or not. Since, $u(z_i)= c\ep$ and $u$ is Lipschitz, Theorem \ref{LipscThm}, we have
  $$d(z_i, \partial \Omega(u)) \geq \frac{c\ep}{C},$$
  and therefore $$d(z_1,z_2)\geq d(z_i, \partial \Omega(u)) \geq \frac{c\ep}{C}.$$
  Hence, $\tilde{t}_1$ and $\tilde{t}_2$ are different than $1$ for any $r \in [2^{-l},r^*].$ Contradiction, thus the result follows.
 \par [Proof of the second part](By way of contradiction) Assume that, for every $h_0>0$, there exists $h<h_0$, such that we have
$$\{x\in B_{1/2}(0)| B_h(x) \subseteq \{x|u(x)=0\}^o \} = \emptyset.$$
Let us construct a larger subsolution under the above assumptions. Let $x \in [B_{1/2}(0)\backslash B_{1/2-2h}(0)]\cap \{x|u(x)=0\}^o$, then $B_h(x)\cap \Omega(u) \neq \emptyset$. We claim that there exists a point $z \in  [B_{1/2}(x)\backslash B_{1/2-2h}(0)]\cap \{x|u(x)=0\}^o$ such that $B_h(z)$ intersects with $\Omega(u)$ in two components, as shown in Figure \ref{resimball}.
\begin{figure}
  \includegraphics[width=2.5in]{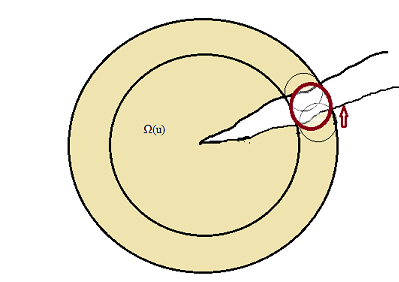}\\
  \caption{Rotate balls of size $h$ counterclockwise}\label{resimball}
\end{figure}
 In order to catch such a ball, we start from the ball $B_h(x)$, say $x=(1/2-h, \theta)$ in polar coordinates for some $\theta\in [0,2\pi)$, and rotate this ball in counterclockwise direction. There exists a point $z \in [B_{1/2}(0)\backslash B_{1/2-2h}(0)]\cap \{x|u(x)=0\}^o$ such that $B_h(z)$ is intersected with $\Omega(u)$ in two components, otherwise we can insert a ball $B_h(z_0)$ into $\{x|u(x)=0\}^o$ which contradicts to our assumption. Hence, consider $B_h(z)$ such that it is intersected with $\Omega(u)$ in two components. Then there exists $r\in [1/2-2h,1/2]$ and two points $x_1$ and $x_2$ in $B_h(z)$ such that $|x_1|=|x_2|=r$ and these points are in separate components of $\Omega(u)$. Now, we can construct a larger subsolution for sufficiently small $h>0$ as follows: Consider the harmonic function $h(x)$ in $B_r(0)$ with $w=u$ in $B_{1/2}(0)\setminus B_r(0)$. Then consider $w=max\{u,h\}$ in $B_{1/2}(0)$. Thus, $w$ becomes a larger subsolution than $u$: as we showed previously, this claim is true if the FBC, $(iii)$ in Definition \ref{Def1}, is satisfied by $w$; so if $x_0 \in \partial\Omega(w)$ with a tangent ball from outside, $x_0$ should be on $\partial B_{r}(0)\cap B_{h}(z)$ and by Lemma \ref{biggradient for2planes} we have $w_\eta(x_0) \geq \frac{C}{h}$ where $h$ is the distance between two components of $\Omega(u)$ in $B_h(z)$ which is sufficiently small this time. Hence, $w$ is a larger subsolution than $u$. Contradiction, hence we get the result.
\end{proof}
At the beginning of this section, we normalized a neighborhood, $B_r(x_0)$, of the free boundary which contains components of $\Omega(u)$ up to the radius $B_{r/2}(x_0)$. On the other hand, by normalization and Corollary \ref{onecomponent}, this neighborhood can be characterized with only two components as $\Omega(u)$ and its complement.
\\
\appendix
\input{appendix}

\begin{flushleft}
\textsl{\textbf{Acknowledgements}}
\end{flushleft}

I owe my deepest gratitude to my graduate advisor Luis A. Caffarelli for his supervision and support during my graduate degree at The University of Texas at Austin. His unique perspectives and visualizations for the subjects were the most significant contribution to this paper. I am also thankful to my colleagues Alessio Figalli, Nestor Guillen, and Ray Yang for their constructive feedback and productive exchange of ideas on related subjects.

\bibliographystyle{plain}
\bibliography{refs}

\end{document}

%% file: appendix.tex
\section{}
Let us remind you the Monotonicity formula for
$\mathbb{R}^2$, the reader can consult to \cite{CaffSalsa} for detailed theory:
\begin{thm}\label{MonoctyThm}[Monotonicity Theorem]\index{Theorem! Monotonicity @\emph{Monotonicity}} Let $B_1(0) \in
\mathbb{R}^2$ and $u_1$, $u_2$ $\in \mathrm{H}^1(B_2(0))$, continuous and
nonnegative in $B_2(0)$, supported and harmonic in disjoint domains
$\Omega_1$, $\Omega_2$, respectively, with $0\in\partial\Omega_i$ and
 $$u_i=0 \; \text{along}\; \partial\Omega_i \cap B_1=\Gamma_i \;\; (i=1,2).$$
 Then the quantity
 $$J(R)=\ds \frac{1}{R^4}\int_{B_R(0)}|\nabla u_1|^2dx \cdot\int_{B_R(0)}|\nabla
u_2|^2dx$$
 is monotone increasing  in $R$, $R \leq 3/2$.
 \end{thm}
 \begin{proof}
   We want to show that $J'(R)\geq 0$ a.e. $R \in (0,3/2)$. By rescaling, it is enough to prove that $J'(1)\geq 0$.
   Observe that, $$
   \frac{d}{dr}\int_{B_r(0)}|\nabla u_i|^2dx= \int_{\partial B_r(0)}|\nabla u_i|^2d \sigma \in \mathrm{L}^1(0,2)
   $$
   and
   $$J'(1)= \int_{\partial B_1(0)}|\nabla u_1|^2d \sigma \cdot\int_{B_1(0)}|\nabla
u_2|^2dx + \int_{B_1(0)}|\nabla u_1|^2dx \cdot\int_{\partial B_1(0)}|\nabla
u_2|^2d \sigma - 4\int_{B_1(0)}|\nabla u_1|^2dx \cdot\int_{B_1(0)}|\nabla
u_2|^2dx.$$
Then, we get

$$
\begin{array}{lcl}
  \ds \frac{J'(1)}{J(1)}&=& \ds \frac{\int_{\partial B_1(0)}|\nabla u_1|^2d \sigma}{\int_{B_1(0)}|\nabla u_1|^2dx} + \ds \frac{\int_{\partial B_1(0)}|\nabla
u_2|^2d \sigma}{\int_{B_1(0)}|\nabla
u_2|^2dx} - 4.
\end{array}
$$
Since $u_i$ is harmonic and supported in $\Omega_i$, we have $\Delta u_i^2= 2 |\nabla u_i|^2$ which implies

$$
\begin{array}{lcl}
  \ds\int_{B_1(0)}|\nabla u_i|^2dx & =& \ds\int_{\partial B_1(0)} u_i (u_i)_r d \sigma \\
  & =&\ds \left( \int_{\partial B_1(0)} u_i^2  d \sigma\right)^{1/2} \ds\left(\int_{\partial B_1(0)} u_i^2 d \sigma \right)^{1/2}
\end{array}
$$
where $u_r$ denotes the exterior radial derivative of $u$ along $\partial B_1(0)$. Let us denote $u_\theta$ as the tangential derivative of $u$ along $\partial B_1(0)$ then we get
$$
\begin{array}{lcl}
  \ds\int_{\partial B_1(0)}|\nabla u_i|^2d\sigma& \geq & 2 \ds \left( \int_{\partial B_1(0)} (u_i)_r^2  d \sigma\right)^{1/2} \ds\left(\int_{\partial B_1(0)} (u_i)_\theta^2 d \sigma \right)^{1/2}.
\end{array}
$$
Hence, it is enough to prove that

$$
\begin{array}{lcl}
  \ds \frac{J'(1)}{J(1)}&\geq& \ds \frac{\left(\int_{\partial B_1(0)}(u_1)_\theta^2d \sigma\right)^{1/2}}{\left(\int_{\partial B_1(0)} u_1^2d\sigma\right)^{1/2}} + \ds \frac{\left(\int_{\partial B_1(0)}
(u_2)^2_\theta d \sigma\right)^{1/2}}{\left(\int_{\partial B_1(0)}
u_2^2d \sigma\right)^{1/2}} - 2 \geq 0.
\end{array}
$$
Thus, if we can estimate the minimum of the quotient
$$
\ds\frac{\int_{\Gamma_i}(u_i)_\theta^2d \sigma}{\int_{\Gamma_i} u_i^2d\sigma},
$$ 
then we will obtain the result.
These quotients are minimized by the first eigenfunction of the domains $\partial B_1(0) \cap \overline{\Omega_i}$, respectively. Moreover, since $\Omega_1 \cap\Omega_2=\emptyset$, the question is reduced to find the minimizer of the quotient
$$
\inf_{v\in H_0^1(\Gamma)}\ds\frac{\int_{\Gamma}(v)_\theta^2d \sigma}{\int_{\Gamma} v^2d\sigma}
$$ 
for a given domain $\Gamma$ in $B_1(0)$ with a measure $\mu$, i.e. to find $v\in H_0^1(\Gamma)$ which has the smallest eigenvalue. We obtain, by the symmetrization argument, the optimal domain as a connected arc with the larger the arc the smaller the quotient. Thus, when we consider the domains $\Gamma_1$ and $\Gamma_2$, and the sum 

$$
\ds \frac{\left(\int_{\Gamma_1}(u_1)_\theta^2d \sigma\right)^{1/2}}{\left(\int_{\Gamma_1} u_1^2d\sigma\right)^{1/2}} + \ds \frac{\left(\int_{\Gamma_2}
(u_2)^2_\theta d \sigma\right)^{1/2}}{\left(\int_{\Gamma_2}
u_2^2d \sigma\right)^{1/2}} ,
$$
then this sum attains its minimum for two adjacent, complementary arcs with $u_1$ and $u_2$ the corresponding eigenfunctions. If the arcs have length $\alpha 2\pi$ and $(1-\alpha)2\pi$, then the corresponding eigenfunctions are 
$$\begin{array}{ccc}
    \sin\frac{\theta}{2\alpha} & \text{and}& \sin\frac{\theta}{2(1-\alpha)}
  \end{array}
$$
and the sum 
 $$
\ds \frac{\left(\int_{\Gamma_1}(u_1)_\theta^2d \sigma\right)^{1/2}}{\left(\int_{\Gamma_1} u_1^2d\sigma\right)^{1/2}} + \ds \frac{\left(\int_{\Gamma_2}
(u_2)^2_\theta d \sigma\right)^{1/2}}{\left(\int_{\Gamma_2}
u_2^2d \sigma\right)^{1/2}}\geq \frac{1}{2\al}+ \frac{1}{2(1-\al)}\geq 2 , \;\text{for}\;\al \in [0,1]
$$
which implies the result.

 \end{proof}